\documentclass[11pt,a4paper]{article}

\usepackage{epsf,epsfig,amsfonts,amsgen,amsmath,amstext,amsbsy,amsopn,amsthm
}
\usepackage{ebezier,eepic}
\usepackage{color}

\newtheorem{theorem}{Theorem}

\newtheorem{prop}{Proposition}
\newtheorem{lemma}{Lemma}
\newtheorem{false statement}{False statement}

\theoremstyle{definition}

\newtheorem{claim}{Claim}

\newtheorem{conjecture}{Conjecture}

\newtheorem{problem}{Problem}
\newtheorem{case}{Case}
\newtheorem{subcase}{Case}[case]
\newtheorem{subsubcase}{Case}[subcase]

\newcounter{mathitem}
  {\begin{list}{{$(\roman{mathitem})$}}{
   \setcounter{mathitem}{0}
   \usecounter{mathitem}
   \setlength{\topsep}{0pt plus 2pt minus 0pt}
   \setlength{\parskip}{0pt plus 2pt minus 0pt}
   \setlength{\partopsep}{0pt plus 2pt minus 0pt}
   \setlength{\parsep}{0pt plus 2pt minus 0pt}
   \setlength{\leftmargin}{35pt}
   \setlength{\itemsep}{0pt plus 2pt minus 0pt}}}
  {\end{list}}

\begin{document}

\title{\bf\Large A note on heterochromatic cycles of length 4 in edge-colored graphs\thanks{Supported by NSFC
(No.~10871158).}}

\date{}

\author{Bo Ning and Shenggui Zhang\thanks{Corresponding author. E-mail address: sgzhang@nwpu.edu.cn (S. Zhang).}    \\[2mm]
\small Department of Applied Mathematics,  \small Northwestern Polytechnical University,\\
\small Xi'an, Shaanxi 710072, P.R.~China\\} \maketitle

\begin{abstract}
Let $G$ be an edge-colored graph. A heterochromatic cycle of $G$ is one in which every two edges have different colors. For a vertex $v\in V(G)$, let $CN(v)$ denote the set of colors which are assigned to the edges incident to $v$. In this note we prove that $G$ contains a heterochromatic cycle of length 4 if $G$ has $n\geq 60$ vertices and $|CN(u)\cup CN(v)|\geq n-1$ for every pair of vertices $u$ and $v$ of $G$. This extends a result of Broersma et al. on the existence of heterochromatic cycles of length 3 or 4.

\smallskip
\end{abstract}

\section{Introduction}

We use Bondy and Murty \cite{Bondy_Murty_0} for terminology and
notation not defined here and consider finite simple graphs only.

Let $G=(V,E)$ be a graph, $H$ a subgraph, and $v$ a vertex of $G$. We use $N_H(v)$ to denote the set, and $d_H(v)$ the number, of neighbors of $v$ in $H$, and call $d_H(v)$ the \emph{degree} of $v$ in $H$. An \emph{edge-coloring} of $G$ is a function $C : E\longrightarrow \mathbb N^{+}$, where $\mathbb N^{+}$ is the set of positive integers. We call $G$ an \emph{edge-colored graph} if
it is assigned such a coloring $C$, and denote this edge-colored graph by $(G,C)$. We use $CN_{H}(v)$ to denote the set of colors of the edges joining $v$ and its neighbors in $H$ when $v\notin V(H)$ and those incident to $v$ in $H$ when $v\in V(H)$, and call $d^{c}_{H}(v)=|CN_{H}(v)|$ the \emph{color degree} of $v$ in $H$. When no confusion occurs, we use $G$ instead of $(G,C)$, and $N(v)$, $d(v)$, $CN(v)$ and $d^{c}(v)$ instead of $N_G(v)$, $d_G(v)$, $CN_{G}(v)$ and $d^{c}_{G}(v)$, respectively.

A subgraph $H$ of an edge-colored graph $G$ is called \emph{heterochromatic} if every two of its edges have different colors. Recently, heterochromatic subgraphs in edge-colored graphs have received much attentions. Heterochromatic matchings were studied in \cite{Lesaulnier_Stocker_Wenger_West,Wang,Wang_Li}. Chen and Li \cite{Chen_Li_0,Chen_Li_5} studied long heterochromatic paths in edge-colored graphs. Albert et al. \cite{Albert_Frieze_Reed} studied heterochromatic Hamilton cycles in edge-colored complete graphs. Heterochromatic cycles of small lengths were studied in \cite{Broersma_Li_Woegingerr_Zhang,Li_Wang,Wang_Li_Zhu_Liu,Zhu}. For a survey on the study of heterochromatic subgraphs in edge-colored graphs, we refer to \cite{Kano_Li}.

In the following we use $C_k$ to denote a cycle of length $k$. Broersma et al. \cite{Broersma_Li_Woegingerr_Zhang} considered the existence of short heterochromatic cycles in edge-colored graphs and got the following result.

\begin{theorem}[Broersma, Li, Woeginger and Zhang \cite{Broersma_Li_Woegingerr_Zhang}]\label{t1}
Let $G$ be an edge-colored graph of order $n\geq 4$ such that $|CN(u)\cup CN(v)|\geq n-1$ for every pair of vertices $u$ and $v$ of $G$. Then $G$ contains a heterochromatic $C_3$ or a heterochromatic $C_4$ .
\end{theorem}

Li and Wang \cite{Li_Wang} gave a result on the existence of a heterochromatic $C_3$ or a heterochromatic $C_4$ under the color degree condition. At the same time, they obtained a result on the existence of a heterochromatic $C_3$ in edge-colored graphs.

\begin{theorem}[Li and Wang \cite{Li_Wang}]\label{t2}
Let $G$ be an edge-colored graph of order $n\geq 3$. If $d^{c}(v)\geq (\frac{4\sqrt{7}}{7}-1)n+3-\frac{4\sqrt{7}}{7}$ for every vertex $v\in V(G)$, then $G$ contains either a heterochromatic $C_3$ or a heterochromatic $C_4$.
\end{theorem}

\begin{theorem}[Li and Wang \cite{Li_Wang}]\label{t3}
Let $G$ be an edge-colored graph of order $n\geq 3$. If $d^{c}(v)\geq \frac{\sqrt{7}+1}{6}n$ for every vertex $v\in V(G)$, then $G$ contains a heterochromatic $C_3$.
\end{theorem}

Later, Wang et al. \cite{Wang_Li_Zhu_Liu} got a result on the existence of a heterochromatic $C_4$ in edge-colored triangle-free graphs. Zhu \cite{Zhu} further extended this result to edge-colored bipartite graphs.

\begin{theorem}[Wang, Li, Zhu and Liu \cite{Wang_Li_Zhu_Liu}]\label{t4}
Let $G$ be an edge-colored triangle-free graph of order $n\geq 9$. If $d^{c}(v)\geq \frac{3-\sqrt{5}}{2}n+1$ for every vertex $v\in V(G)$, then $G$ contains a heterochromatic $C_4$.
\end{theorem}

\begin{theorem}[Zhu \cite{Zhu}]\label{t5}
Let G be an edge-colored bipartite graph of order $n\geq 6$. If $d^{c}(v)\geq \frac{(\sqrt{5}-1)n}{4}+1$ for every vertex $v\in V(G)$, then $G$ contains a heterochromatic $C_4$.
\end{theorem}

In this note we prove the following result which is an extension of Theorem \ref{t1}.

\begin{theorem}\label{t6}
Let $G$ be an edge-colored graph of order $n\geq 60$ such that $|CN(u)\cup CN(v)|\geq n-1$ for every pair of vertices $u$ and $v$ of $G$. Then $G$ contains a heterochromatic $C_4$.
\end{theorem}

We can not provide examples to show that the restriction $n\geq 60$ can be reduced and the condition $|CN(u)\cup CN(v)|\geq n-1$ can be weakened in Theorem \ref{t6}. However, we can show that the lower bound of $|CN(u)\cup CN(v)|$ must be more than $\sqrt{2n-3}$ in order to guarantee the existence of a heterochromatic $C_4$.

A finite projective plane $\mathcal P$ is a pair of sets $(P,L)$ where $P$ is a set of points and $L$ is a set of lines such that\\
($i$)  any two points in $P$ lie on only one line in $L$;\\
($ii$) any two lines in $L$ meet in only one point in $P$; and\\
($iii$) there are four points in $P$ no three of which lie on a line in $L$.\\
It is proved that (see \cite{Bondy_Murty_5}, Exercise 1.3.13) for any integer $t$, there exists a finite projective plane $(P,L)$ such that each point in $P$ lies on $t+1$ lines in $L$, each line in $L$ contains $t+1$ points in $P$, and $|P|=|L|=t^2+t+1$.

Let $\mathcal P$ be a finite projective plane with $|P|=|L|=t$. We define the {\em incidence graph} $I(\mathcal P)$ of $\mathcal P$ as the bipartite graph with bipartition $(P,L)$ and edges $pl$ when $p\in P$ lies on $l\in L$ in $\mathcal P$. Let $u,v$ be two vertices of $I(\mathcal P)$. If $u,v\in P$ or $u,v\in L$, then $|N(u)\cup N(v)|=2t+1$. If one of them (say $u$) is in $P$ and the other ($v$) is in $L$, then $|N(u)\cup N(v)|=2t+1$ when $u$ lies on $v$ and $|N(u)\cup N(v)|=2t+2$ otherwise. Therefore, we have
$|N(u)\cup N(v)|\geq 2t+1= \sqrt{2n-3}$. On the other hand, it is easy to see that $I(\mathcal P)$ contains no cycles of length 4. Otherwise, there exist two points in $P$ which lie on two lines, contradicting to the definition of finite projective planes.

If we assign an edge coloring to $I(\mathcal P)$ such that every two of its edges have different colors, then we can get the following proposition.

\begin{prop}
Let $t$ be an integer. Then there exists an edge-colored graph $G$ of order $n=2(t^2+t+1)$ such that $|CN(u)\cup CN(v)|\geq\sqrt{2n-3}$ for every pair of vertices $u$ and $v$ of $G$ and $G$ contains no heterochromatic cycles of length 4.
\end{prop}

\section{Proof of Theorem \ref{t6}}
We first give a lemma, which will be used in the proof of Theorem \ref{t6}.

\begin{lemma}\label{t7}
Let $G$ be an edge-colored graph. Then $G$ contains a spanning bipartite subgraph $H$ such that $2d_{H}^{c}(v)+3d_{H}(v)\geq d_G^{c}(v)+d_G(v)$ for every vertex $v\in V(H)$.
\end{lemma}

\begin{proof}
We choose a spanning bipartite subgraph $H$ of $G$ such that $f(H)=|E(H)|+\sum_{v\in V(H)}d_{H}^{c}(v)$ is as large as possible and show that $2d_{H}^{c}(v)+3d_{H}(v)\geq d_G^{c}(v)+d_G(v)$ for every vertex $v\in V(H)$.

Assume that the bipartition of $H$ is $(X,Y)$. Then any edge $xy$ of $G$ with $x\in X$ and $y\in Y$ is also an edge of $H$. Otherwise, we have $f(H+xy)> f(H)$, contradicting to the choice of $H$. It can be seen that $d^{c}_{H}(x)=|CN_{G[Y]}(x)|$ for $x\in X$ and $d^{c}_{H}(y)=|CN_{G[X]}(y)|$ for $y\in Y$.

Suppose that there exists a vertex $w\in V(H)$ such that
\begin{align}
2d_{H}^{c}(w)+3d_{H}(w)<d_G^{c}(w)+d_G(w).
\end{align}
Without loss of generality, we assume $w\in X$. Let $H'$ be the spanning bipartite subgraph of $G$ with bipartition $(X\backslash \{w\}, Y\cup \{w\})$ and edge set $E(H)\cup \{wx: x\in X\setminus \{w\}\}\setminus \{wy: y\in Y\}$. Then, we have
\begin{align}
|E(H')|-|E(H)|=(d_{G}(w)-d_{H}(w))-d_{H}(w)=d_{G}(w)-2d_{H}(w).
\end{align}
On the other hand, noting that
\begin{align*}
d^{c}_{H'}(w)-d^{c}_{H}(w)
&=|CN_{G[X]}(w)|-|CN_{G[Y]}(w)|\\
&=|CN_{G\setminus G[Y]}(w)|-|CN_{G[Y]}(w)|\\
&\geq |CN_{G}(w)|-2|CN_{G[Y]}(w)|\\
&=d^{c}_{G}(w)-2d^{c}_{H}(w)
\end{align*} and
\begin{align*}
\sum_{v\in V\setminus\{w\}}(d^{c}_{H'}(v)-d^{c}_{H}(v))
&=\sum_{v\in X\setminus\{w\}}(d^{c}_{H'}(v)-d^{c}_{H}(v))+\sum_{v\in Y}(d^{c}_{H'}(v)-d^{c}_{H}(v))\\
&\geq \sum_{v\in Y}(d^{c}_{H'}(v)-d^{c}_{H}(v))\\
&= \sum_{v\in Y} (|CN_{G[X\backslash \{w\}]}(v)|-|CN_{G[X]}(v)|)\\
&\geq -\sum_{v\in Y} |CN_{G[\{w\}]}(v)|\\
&=-d_{H}(w),
\end{align*}
we have
\begin{align*}
\sum_{v\in V}d^{c}_{H'}(v)-\sum_{v\in V}d^{c}_{H}(v)
&=\sum_{v\in V\setminus\{w\}}(d^{c}_{H'}(v)-d^{c}_{H}(v))+(d^{c}_{H'}(w)-d^{c}_{H}(w))\\
&\geq (d^{c}_{G}(w)-2d^{c}_{H}(w))-d_{H}(w).
\end{align*}
That is,
\begin{align}
\sum_{v\in V}d^{c}_{H'}(v)-\sum_{v\in V}d^{c}_{H}(v)\geq  d^{c}_{G}(w)-2d^{c}_{H}(w)-d_{H}(w).
\end{align}
By (1), (2) and (3), we get
$$
f(H')-f(H)\geq d_{G}(w)+d^{c}_{G}(w)-2d^{c}_{H}(w)-3d_{H}(w)>0,
$$
which contradicts to the choice of $H$.

The proof is complete.
\end{proof}

\noindent{}
{\bf {Proof of Theorem \ref{t6}}}

We denote $\delta^{c}(G)=\min\{d^{c}(v):v\in V(G)\}$ and distinguish two cases.
\begin{case}\label{t8}
$\delta^{c}(G)=n-1$.
\end{case}

Assume that $G$ contains no heterochromatic cycles of length 4. Then we have $d^{c}(v)=d(v)=n-1$ for every vertex $v\in V(G)$. Hence $G$ is complete and $C(vx)\neq C(vy)$ for every vertex $v\in V(G)$, where $x,y\in N(v)$.

\begin{claim}\label{t9}
Let $x_1,x_2,x_3,x_4$ be four vertices of $V(G)$. If $C(x_1x_2)\neq C(x_3x_4)$, then $C(x_2x_3)=C(x_1x_4)$.
\end{claim}

\begin{proof}
Suppose that $C(x_2x_3)\neq C(x_1x_4)$. From the above discussion, we know that $C(x_1x_2)\neq C(x_3x_4)$, $C(x_1x_2)\neq C(x_2x_3)$, $C(x_2x_3)\neq C(x_3x_4)$, $C(x_3x_4)\neq C(x_4x_1)$ and $C(x_4x_1)\neq C(x_1x_2)$. With $C(x_2x_3)\neq C(x_1x_4)$, we see that $x_1x_2x_3x_4x_1$ is a heterochromatic $C_4$, a contradiction.
\end{proof}

Let $y_1,y_2,y_3,y_4,y_5$ be five vertices of $V(G)$. Without loss of generality, we assume $C(y_5y_k)=k$ for $k\in\{1,2,3,4\}$. Since $C(y_1y_2)\neq C(y_1y_5) $ and $C(y_1y_2)\neq C(y_2y_5)$, we have $C(y_1y_2)\neq 1$ and $C(y_1y_2)\neq 2$. Hence $C(y_1y_2)\geq 3$.

Suppose that $C(y_1y_2)=3$.  Since $C(y_1y_2)\neq C(y_4y_5)$, it follows from Claim \ref{t9} that $C(y_2y_4)=C(y_5y_1)$. With $C(y_3y_5)\neq C(y_1y_5)$, we have $C(y_2y_4)\neq C(y_3y_5)$. Then by Claim \ref{t9}, we obtain that $C(y_3y_4)=C(y_2y_5)$. At the same time, by $C(y_1y_2)\neq C(y_4y_5)$ and Claim \ref{t9}, we have $C(y_1y_4)=C(y_2y_5)$. Hence $C(y_1y_4)=C(y_3y_4)$, a contradiction.

Suppose that $C(y_1y_2)=4$. Since $C(y_1y_2)\neq C(y_3y_5)$, it follows from Claim \ref{t9} that $C(y_2y_3)=C(y_1y_5)$. With $C(y_4y_5)\neq C(y_1y_5)$, we get $C(y_2y_3)\neq C(y_4y_5)$. Then by Claim \ref{t9}, we obtain that $C(y_3y_4)=C(y_2y_5)$. Similarly, with $C(y_1y_2)\neq C(y_2y_5)$, we have $C(y_1y_2)\neq C(y_3y_4)$. Then by Claim \ref{t9}, we have $C(y_1y_4)=C(y_2y_3)$. Hence, $C(y_1y_4)=C(y_1y_5)$, a contradiction.

Suppose that $C(y_1y_2)\geq 5$. Since $C(y_1y_2)\neq C(y_4y_5)$, it follows from Claim \ref{t9} that $C(y_2y_4)=C(y_1y_5)$. At the same time, by $C(y_1y_2)\neq C(y_3y_5)$ and Claim \ref{t9}, we get $C(y_2y_3)=C(y_1y_5)$. Hence $C(y_2y_4)=C(y_2y_3)$, a contradiction.

The proof of Case \ref{t8} is complete.

\begin{case}
$\delta^{c}(G)\leq n-2$.
\end{case}

Let $w$ be a vertex with $d^{c}(w)=\delta^{c}(G)$ and denote $\delta^{c}(G)=k$. Let $T$ be a subset of $N(w)$ such that $|T|=k$ and $C(wx)\neq C(wy)$ for every two vertices $x,y\in T$. Without loss of generality, set $T(w)=\{x_1,x_2,\ldots,x_k\}$ and assume that $C(wx_i)=i$ for $i\in \{1,2,\ldots,k\}$. Set $G_1=G[T\cup \{w\}]$ and $G_2=G[V(G)\setminus V(G_1)]$. Since $|V(G_1)|=k+1\leq n-1$, we have $V(G_2)\neq \emptyset$.

\begin{subcase}
There exits a vertex $z\in V(G_2)$ such that $|CN_{G_1}(z)\setminus CN(w)|\geq 2$.
\end{subcase}

By the choice of $T$, if $v$ is a neighbor of $z$ such that $C(vz)\in CN_{G_1}(z)\setminus CN(w)$, then $v\neq w$. Since $|CN_{G_1}(z)\setminus CN(w)|\geq 2$, we can choose $x_s,x_t\in T$ with $\{C(x_sz),C(x_tz)\} \subseteq CN_{G_1}(z)\setminus CN(w)$. Obviously, $wx_szx_tw$ is a heterochromatic $C_4$.

\begin{subcase}
$|CN_{G_1}(v)\setminus CN(w)|\leq 1$ for every vertex $v\in V(G_2)$.
\end{subcase}

\begin{claim}
$|CN_{G_2}(v)|=|V(G_2)|-1$ for every vertex $v\in V(G_2)$.
\end{claim}
\begin{proof}
First, we have $|CN_{G_1}(v)\setminus CN(w)|\leq 1$. If follows from $|CN(w)|=k$ that $|CN(w)\cup CN_{G_1}(v)|\leq k+1$. Note that $|CN(w)\cup CN(v)|\geq n-1$, we have $|CN(v)\backslash CN_{G_1}(v)|\geq n-k-2$. On the other hand, we have $|CN(v)\backslash CN_{G_1}(v)|\leq |CN_{G_2}(v)|\leq d_{G_2}(v)\leq |V(G_2)|-1=n-k-2$. Thus, $|CN_{G_2}(v)|=|V(G_2)|-1$, where $|V(G_2)|=n-k-1$.
\end{proof}

\begin{subsubcase}
$k\leq n-6$.
\end{subsubcase}
Note that $|V(G_2)|=n-k-1\geq 5$. Then as in Case \ref{t8}, we can prove that there exists a heterochromatic $C_4$ in $G_2$, which is also a heterochromatic $C_4$ in $G$.
\begin{subsubcase}
$k\geq n-5$.
\end{subsubcase}
By Lemma \ref{t7}, there is a spanning bipartite subgraph $H$ such that
\begin{align}
2d_{H}^{c}(v)+3d_{H}(v)\geq d_G^{c}(v)+d_G(v)
\end{align}
for every vertex $v\in V(H)$. It is not difficult to see that
\begin{align}
d_{H}(v)-d_{H}^{c}(v)\leq d_{G}(v)-d_{G}^{c}(v)
\end{align}
and
\begin{align}
d_{G}(v)-d_{G}^{c}(v)\leq d_{G}(v)-\delta^{c}(G)\leq (n-1)-(n-5)=4.
\end{align}
Together with (5) and (6), we have
\begin{align}
d_{H}^{c}(v)-d_{H}(v)\geq -4.
\end{align}
Then, combining (4) with (7), we obtain
$$
d^{c}_{H}(v)\geq \frac{1}{5}(d_G^{c}(v)+d_G(v)-12)\geq \frac{2n-22}{5}>\frac{(\sqrt{5}-1)n}{4}+1
$$
when $n\geq 60$. By Theorem \ref{t5}, there is a heterochromatic $C_4$ in $H$, which is also a heterochromatic $C_4$ in $G$.

The proof is complete.
{\hfill$\Box$}

\section{Remarks}
Our proof of Theorem \ref{t6} mainly relies on Lemma \ref{t7} and Theorem \ref{t5}.

Lemma \ref{t7} was motivated by the following result due to Erd\"{o}s \cite{Erdos}.

\begin{theorem}[Erd\"{o}s \cite{Erdos}]
Let $G$ be a graph. Then $G$ contains a spanning bipartite subgraph $H$ such that $d_{H}(v)\geq \frac{1}{2}d_{G}(v)$
for all $v\in V(H)$.
\end{theorem}

We have the following problem.

\begin{problem}\label{t9}
Let $G$ be an edge-colored graph. Does $G$ contain a spanning bipartite subgraph $H$ such that $d^{c}_{H}(v)\geq \frac{1}{2}d^{c}_{G}(v)$?
\end{problem}

Zhu \cite{Wang_Li_Zhu_Liu} pointed out that if the answer to the following conjecture is positive, then it would result in an improvement to Theorem \ref{t5}.

\begin{conjecture}[\cite{Wang}]\label{t10}
Let $D$ be a directed bipartite graph with bipartition $(A,B)$. If $d^{+}(u)>\frac{|B|}{3}$ for $u\in A$ and $d^{+}(v)\geq\frac{|A|}{3}$ for $v\in B$, or $d^{+}(u)\geq\frac{|B|}{3}$ for $u\in A$ and $d^{+}(v)>\frac{|A|}{3}$ for $v\in B$, then there exists a directed $C_4$ in $D$.
\end{conjecture}

From our proof of Theorem \ref{t6}, we know that if the answer to Problem \ref{t9} or Conjecture \ref{t10} is positive, then the restriction $n\geq 60$ in Theorem \ref{t6} can be reduced.

\end{document}